\documentclass[12pt]{article}

\usepackage[english]{babel}
\usepackage[utf8]{inputenc}
\usepackage[T1]{fontenc}
\usepackage{amsmath}
\usepackage{amsthm}
\usepackage{amssymb}
\usepackage{graphicx}
\usepackage{enumitem}

\usepackage{tikz}
\usetikzlibrary{arrows}

\tikzstyle{sommetb}=[circle,draw,scale=0.5,fill=black]
\tikzstyle{sommeti}=[circle,draw,scale=0.3,fill=black]

\newtheorem{definition}{Definition}
\newtheorem{prop}{Proposition}

\newtheorem{remark}{Remark}
\newtheorem{exe}{Example}
\newtheorem{lemma}{Lemma}
\newtheorem{theorem}{Theorem}
\newtheorem{cor}{Corollary}

\title{Isoperimetric upper bound for the first eigenvalue of discrete Steklov problems}
\author{Hélène Perrin}
\date{}

\begin{document}
\maketitle

\begin{abstract}

We study upper bounds for the first non-zero eigenvalue of the Steklov problem defined on finite graphs with boundary. For finite graphs with boundary included in a Cayley graph associated to a group of polynomial growth, we give an upper bound for the first non-zero Steklov eigenvalue depending on the number of vertices of the graph and of its boundary. As a corollary, if the graph with boundary also satisfies a discrete isoperimetric inequality, we show that the first non-zero Steklov eigenvalue tends to zero as the number of vertices of the graph tends to infinity. This extends recent results of Han and Hua, who obtained a similar result in the case of $\mathbb{Z}^n$. We obtain the result using metric properties of Cayley graphs associated to groups of polynomial growth.

\end{abstract}

\begin{section}{Introduction}

Let $M$ be a compact Riemannian manifold of dimension $n\geq2$ with boundary $\partial M$. The Steklov problem on $M$ is

\[
\begin{cases}
\triangle u =0&\text{in } M\\
\frac{\partial u}{\partial n}=\sigma u &\text{on } \partial M
\end{cases}
\]
where $\triangle$ is the Laplace-Beltrami operator and $\frac{\partial u}{\partial n}$ is the outward normal derivative along the boundary $\partial M$. It is a well known result that if the boundary is sufficiently regular, the spectrum of the Steklov problem is discrete and its eigenvalues form a sequence $0=\sigma_0<\sigma_1\leq\sigma_2\leq \dots \nearrow \infty$.

An important question in studying the spectral geometry of the Steklov problem is to maximize its eigenvalues under a constraint on the volume of the boundary or on the volume of the manifold. For simply-connected planar domains of prescribed perimeter, it has been shown by R. Weinstock that the disk maximizes $\sigma_1$ (see \cite{Weinstock54}). For bounded Lipschitz domains of fixed volume in $\mathbb{R}^n$, F. Brock proved that the ball maximizes $\sigma_1$ (see \cite{Brock01}). Several upper bounds have also been obtained for different families of manifolds where the volume or the volume of the boundary is fixed. In 2017, a survey of the literature on this question has been given in \cite{gpsurvey}. More recently, it was shown in \cite{bucur2017weinstock} that a Weinstock-type inequality holds in $\mathbb{R}^n$ in the class of convex sets, that is, that among all bounded convex sets in $\mathbb{R}^n$ with prescribed volume of the boundary, the ball maximizes $\sigma_1$.

In this article, we investigate isoperimetric upper bounds for $\sigma_1$ of the Steklov problem on graphs. The Steklov problem on graphs is a discrete analogue of the Steklov problem and has recently received attention in the literature. In \cite{hua2017} and \cite{PStek2018}, lower bounds for the first non-zero eigenvalue are given. A lower bound for higher eigenvalues is given in \cite{hassannezhad2017higher}. For subgraphs of integer lattices, an upper bound has been obtained by W. Han and B. Hua \cite{hua2019steklov}. In \cite{colbois2016steklov}, a relation between the eigenvalues of the Steklov problem on a manifold and the eigenvalues of a discrete problem is established. Hence, results in the discrete and in the Riemannian settings are closely related and the study of the discrete problem is a possible approach to understand the spectral geometry of the Steklov problem.

%We introduce now some definitions which allow us to state the Steklov problem on graphs.

\begin{definition}
A graph with boundary is a pair $(\Gamma, B)$, where $\Gamma=(V,E)$ is a simple, that is without loops or multiple edges, connected graph and $B\subset V$ is a subset of $V$ such that $B\not = \emptyset$ and $E(B,B)=\emptyset$. We call $B$ the boundary of the graph and $B^c$ the interior.
\label{defgraph}
\end{definition}

In this paper, we always consider graphs with boundary that are finite. The space of all real functions defined on the vertices $V$, denoted by $\mathbb{R}^{V}$, is the Euclidean space of dimension $|V|$. Similarly, the space of real functions defined on the vertices of the boundary, denoted $\mathbb{R}^{B}$, is the Euclidean space of dimension $|B|$.

The Laplacian $\triangle$ of a function $v\in  \mathbb{R}^{V}$ is defined by
\[
(\triangle v)(i)=\sum_{j\sim i}(v(i)-v(j))
\]
where $i\sim j$ signifies that $\{i,j\}\in E$.

A function $v\in \mathbb{R}^V$ is called harmonic if
\[
(\triangle v)(i)=\sum_{j\sim i}(v(i)-v(j))=0\quad \forall i \not\in B.
\]

The normal derivative operator $\frac{\partial v}{\partial n}:\mathbb{R}^V \rightarrow \mathbb{R}^{B}$ is defined by
\[\left(\frac{\partial v}{\partial n}\right)(i)=\sum_{j\in B^c,j\sim i}(v(i)-v(j)) \quad i \in B.
\]

%Since there are never edges between two boundary vertices (see Definition \ref{defgraph}), we remark that $(\frac{\partial v}{\partial n})(i)=(\triangle v)(i)$.

\begin{definition}The Steklov problem on a finite graph with boundary $(\Gamma,B)$ is the eigenvalue problem

\[
\begin{cases}
(\triangle v)(i)=0 &\text{if $i\not\in B$}\\
(\frac{\partial v}{\partial n})(i)=\sigma v(i) &\text{if $i\in B$}
\end{cases}
\]
where $v\not \equiv 0$ and $\sigma$ is a spectral parameter.
\end{definition}

As shown in \cite{PStek2018}, the solutions of this problem coincide with the eigenvalues of the discrete Dirichlet-to-Neumann operator defined in \cite{hua2017}. They form a finite sequence $0=\sigma_0\leq\sigma_1\leq...\leq\sigma_{b-1}$, where $b=|B|$.

We recall that we are interested in upper bounds for $\sigma_1$. Therefore, we will always assume that $|B|>1$ because if not, $\sigma_1$ is not defined. A first remark is that without any additional geometric constraint on $(\Gamma, B)$, $\sigma_1$ may become unbounded. This occurs in the following example.
\begin{exe}
We consider the family of graphs with boundary $\{(G_n,B)\}_{n\in\mathbb{N}}$ as shown in Figure \ref{fig1}, that is, two boundary vertices (the bigger vertices) are connected by $n$ paths of length $2$. By computation, we obtain that $\sigma_1(G_n,B)=n$ and hence $\sigma_1$ tends to $+\infty$ as $n$ tends to $+\infty$.

\begin{figure}[h]
\begin{tikzpicture}[scale=0.6]
\draw(-1,0)node[fill=white]{};
\draw(1,0)node[sommetb]{};
\draw(3,0)node[sommetb]{};
\draw(2,0)node[sommeti]{};
\draw(1,0)--(3,0);
\end{tikzpicture}
\begin{tikzpicture}[scale=0.6]
\draw(0,0)node[fill=white]{};
\draw(1,0)node[sommetb]{};
\draw(3,0)node[sommetb]{};
\draw(2,0)node[sommeti]{};
\draw(2,0.5)node[sommeti]{};
\draw(1,0)--(3,0);
\draw(1,0)--(2,0.5)--(3,0);
\end{tikzpicture}
\begin{tikzpicture}[scale=0.6]
\draw(0,0)node[fill=white]{};
\draw(1,0)node[sommetb]{};
\draw(3,0)node[sommetb]{};
\draw(2,0)node[sommeti]{};
\draw(2,0.5)node[sommeti]{};
\draw(2,-0.5)node[sommeti]{};
\draw(1,0)--(3,0);
\draw(1,0)--(2,0.5)--(3,0);
\draw(1,0)--(2,-0.5)--(3,0);
\end{tikzpicture}
\begin{tikzpicture}[scale=0.6]
\draw(0,0)node[fill=white]{};
\draw(1,0)node[sommetb]{};
\draw(3,0)node[sommetb]{};
\draw(2,0)node[sommeti]{};
\draw(2,0.5)node[sommeti]{};
\draw(2,-0.5)node[sommeti]{};
\draw(2,1)node[sommeti]{};
\draw(1,0)--(3,0);
\draw(1,0)--(2,0.5)--(3,0);
\draw(1,0)--(2,-0.5)--(3,0);
\draw(1,0)--(2,1)--(3,0);
\end{tikzpicture}
\begin{tikzpicture}[scale=0.6]
\draw(0,0)node[fill=white]{};
\draw(1,0)node[sommetb]{};
\draw(3,0)node[sommetb]{};
\draw(2,0)node[sommeti]{};
\draw(2,0.5)node[sommeti]{};
\draw(2,-0.5)node[sommeti]{};
\draw(2,1)node[sommeti]{};
\draw(2,-1)node[sommeti]{};
\draw(1,0)--(3,0);
\draw(1,0)--(2,0.5)--(3,0);
\draw(1,0)--(2,-0.5)--(3,0);
\draw(1,0)--(2,1)--(3,0);
\draw(1,0)--(2,-1)--(3,0);
\end{tikzpicture}
\caption{$(G_1,B), (G_2,B), (G_3,B), (G_4,B)\text{ and }(G_5,B)$.}
\label{fig1}
\end{figure}
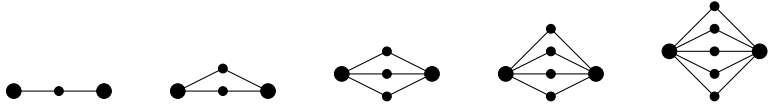
\end{exe}
If we assume that the degree of the graph, $d$, is bounded, it is easy to obtain that $\sigma_1\leq d$.
The goal of this paper is to show isoperimetric upper bounds for the Steklov eigenvalues of graphs with boundary that are included in a Cayley graph with polynomial growth (we recall the notions of geometric group theory that we use in Section 2).

\begin{definition}
A graph with boundary  $(\Gamma'=(V',E'), B)$ is included in a graph $\Gamma=(V,E)$ if $V'\subset V$ and $E'\subset E$.
\end{definition}

\begin{remark} 
The Steklov problem is defined on finite graphs with boundary. In contrast, the Cayley graphs with polynomial growth that we use as host graph are infinite.
\end{remark}

The main result is the following.

\begin{theorem}
Let $\Gamma=(V,E)$ be a Cayley graph with polynomial growth of order $D$. There exists $C(\Gamma)>0$ such that for any finite graph with boundary $(\Gamma'=(V',E'),B)$ included in $\Gamma$ and such that $|B|>1$, we have
\[
\sigma_1(\Gamma',B)\leq
\begin{cases}
C(\Gamma)\frac{1}{|B|}&\text{if }D\leq2,\\
C(\Gamma)\frac{|V'|^{\frac{D-2}{D}}}{|B|}& \text{if } D\geq 2.
\end{cases}
\]
\label{thm1}

\end{theorem}

This result can be pushed further for a particular class of graph with boundary included in the Cayley graph $\Gamma=(V,E)$, \textit{graphs with boundary induced by a subset $\Omega \subset V$}.

\begin{definition}
Let $\Gamma=(V,E)$ be a graph.
\begin{enumerate}
\item The \textit{vertex boundary} of a subset $\Omega \subset \Gamma$ is
$$\delta \Omega:=\{i\in V \setminus\Omega : \exists j \in \Omega,  i\sim j \}$$
where $i\sim j$ signifies that $\{i,j\}\in E$.
\item The set of edges between two subset $\Omega_1, \Omega_2 \subset V$ is 
$$E(\Omega_1,\Omega_2):=\{\{i,j\}\in E : i\in \Omega_1, j\in \Omega_2\}.$$
\item Given $\Omega\subsetneq V$, consider the graph $\Gamma'$ with vertex set $\bar\Omega:=\Omega \cup \delta\Omega$ and edge set $E(\Omega,\bar\Omega)$. This defines a graph with boundary, with $B=\delta\Omega$, which is called \textit{graph with boundary induced by a subset $\Omega \subsetneq V$}.
\item Given a subset $\Omega \subsetneq V$, $\sigma_1(\Omega)$ is the eigenvalue $\sigma_1$ of the graph with boundary induced by $\Omega$.

\end{enumerate}

\end{definition}

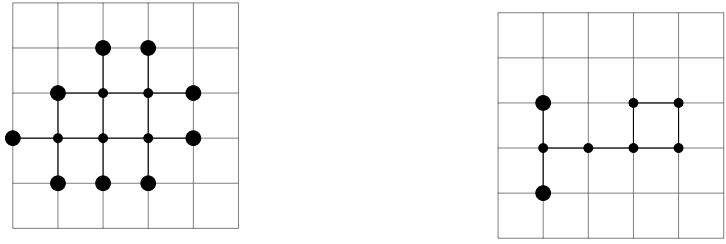
\begin{figure}[h]
\begin{tikzpicture}[scale=0.6]
\draw [very thin, gray] (3,0) grid (8,5);
\draw(0,0)node[fill=white]{};
\draw(3,2)node[sommetb]{};
\draw(4,3)node[sommetb]{};
\draw(5,2)node[sommeti]{};
\draw(5,3)node[sommeti]{};
\draw(4,2)node[sommeti]{};
\draw(6,2)node[sommeti]{};
\draw(6,3)node[sommeti]{};
\draw(4,1)node[sommetb]{};
\draw(5,1)node[sommetb]{};
\draw(6,1)node[sommetb]{};
\draw(7,2)node[sommetb]{};
\draw(7,3)node[sommetb]{};
\draw(6,4)node[sommetb]{};
\draw(5,4)node[sommetb]{};
\draw(3,2)--(7,2);
\draw(4,3)--(7,3);
\draw(4,3)--(4,1);
\draw(5,4)--(5,1);
\draw(6,4)--(6,1);
\end{tikzpicture}
\quad\quad\quad\quad\quad\quad\quad\quad
\begin{tikzpicture}[scale=0.6]
\draw [very thin, gray] (0,0) grid (5,5);
\draw(1,2)node[sommeti]{};
\draw(2,2)node[sommeti]{};
\draw(3,2)node[sommeti]{};
\draw(4,2)node[sommeti]{};
\draw(4,3)node[sommeti]{};
\draw(3,3)node[sommeti]{};
\draw(1,1)node[sommetb]{};
\draw(1,3)node[sommetb]{};
\draw(1,2)--(4,2);
\draw(3,3)--(4,3);
\draw(1,3)--(1,1);
\draw(3,3)--(3,2);
\draw(4,3)--(4,2);
\end{tikzpicture}
\caption{Two graphs with boundary included in $\mathbb{Z}^2$, but only the first one is induced by a subset of vertices of $\mathbb{Z}^2$ (the bigger vertices are boundary vertices).}
\label{fig2}
\end{figure}

Because graphs with boundary induced by a finite subset $\Omega$ of the set of vertices of a Cayley graph with polynomial growth satisfy a discrete isoperimetric inequality, we can deduce the following two corollaries of Theorem \ref{thm1} for this particular case.

\begin{cor}
Let $\Gamma=(V,E)$ be a Cayley graph with polynomial growth of order $D\geq 2$. There exists $C(\Gamma)>0$ such that for any finite connected subset $\Omega$ of the set of vertices $V$ we have
\[
\sigma_1(\Omega)\leq C(\Gamma)\frac{1}{|\delta\Omega|^{\frac{1}{D-1}}}.
\]
\label{cor1}

\end{cor}

\begin{cor}

Let $\Gamma=(V,E)$ be a Cayley graph with polynomial growth of order $D\geq 2$. There exists $C(\Gamma)>0$ such that for any finite connected subset $\Omega$ of the set of vertices $V$ we have
\[
\sigma_1(\Omega)\leq C(\Gamma)\frac{1}{|\bar\Omega|^{\frac{1}{D}}},
\]
where $\bar\Omega=\delta\Omega \cup \Omega$.
\label{cor2}
\end{cor}

A direct consequence is that for a sequence of graphs with boundary induced by subsets in a Cayley graph with polynomial growth such that the number of vertices tends to infinity, $\sigma_1$ tends to zero. It is easy to find examples (see, e.g. Example 1 in \cite{PStek2018}) showing that this is not true if we do not assume that the graphs with boundary are induced by subsets of the Cayley graph.

%A direct consequence is that for a sequence of finite subset $\Omega_k$ in a Cayley graph with polynomial growth such that the number of vertices tends to infinity, $\sigma_1$ of the graphs with boundary induced by $\Omega_k$ tends to zero.

%In a Cayley graph associated to a group of polynomial growth of order $2$, for graphs with boundary that are not necessarily subgraphs and may not satisfy the discrete isoperimetric inequality, we give for $\sigma_1$ an upper bound in terms of the inverse of the number of vertices of the boundary, which proves that for a sequence of such graphs with boundary, $\sigma_1$ tends to $0$ as the number of vertices of the boundary tends to infinity.

%In the case of $\mathbb{Z}^n$ this last result has been showed by Han and Hua (see Corollary 1.4 in \cite{hua2019steklov}). They obtain a similar upper bound using a very interesting method to reduce to the Euclidean case. In addition they are able to give explicit constants. The proof of our main result in the setting of Cayley graphs with polynomial growth is different since it does not use known results for domains in the Euclidean space. We essentially use the control of the growth function of the Cayley graph. The method was inspired by the methods used in \cite{CEG2011}. A straightforward example of a Cayley graph of a group of polynomial growth that is different from $\mathbb{Z}^n$ is a Cayley graph associated to the discrete Heisenberg group of dimension $3$, which has polynomial growth of order $4$. Many other examples exist (see Example \ref{exGrom}) where the result holds.

In $\mathbb{Z}^n$, with $n\geq 2$, Corollary 2 corresponds to a recent result of Han and Hua (see Corollary 1.4 in \cite{hua2019steklov}). They show it using a very interesting method to reduce to the case of domains in $\mathbb{R}^n$, which also allows them to give explicit constants. In the contrast to the proof of the result of Han and Hua, the proof of our main result is direct because it does not use known results for domains in Euclidean space. It essentially uses the control of the growth function of the Cayley graph. The method was inspired by the methods used in \cite{CEG2011}. A straightforward example of a Cayley graph of a group of polynomial growth that is different from $\mathbb{Z}^n$ is a Cayley graph associated to the discrete Heisenberg group of dimension $3$, which has polynomial growth of order $4$. Many other examples exist (see Example \ref{exGrom}) where the result holds.

\end{section}

\begin{section}{Groups with polynomial growth and Cayley graphs}
In this article, we work in the setting of Cayley graphs of groups with polynomial growth. We recall here the definitions and the geometric group theory notions that we will use. For further details on this topic, one can see e.g. \cite{Har00}.

Let $G$ be a finitely generated infinite discrete group and S=$\{g_1,...,g_k\}$ a generating set of $G$. For $n\in \mathbb{N}^*$, we denote the ball of radius n $B(n):=\{x \in G: x=g_{i_1}^{\epsilon_1}...g_{i_n}^{\epsilon_n}, i_1,...,i_n\in\{1,...,k\}, \epsilon_j=\pm 1\}$. The growth function of $G$ is $V(n):=|B(n)|$. If there exist $D\in \mathbb{N}^*$ and $C>0$ such that
\[
C^{-1}n^D\leq V(n)\leq Cn^D,
\]
we say that the growth rate is polynomial of order $D$. Since the growth rate does not depend on the choice of generating set, we can speak of the growth type of a group.

Let $G$ be a group and $S$ a generating set that does not contain the identity element of the group and is symmetric, that is, satisfies $S=S^{-1}$. The Cayley graph $\Gamma=\Gamma(G,S)$ associated to $(G,S)$ is the graph with vertices $V=G$ and edges $E=\{\{x,y\}: x,y\in V \text{ and } \exists s\in S\text{ such that } y=xs\}$. Since $S$ is symmetric and does not contain the identity element, the graph is simple, and since $S$ is a generating set of $G$, the graph is connected. We say that a Cayley graph has polynomial growth of order $D$ if it is associated to a group with polynomial growth of order $D$.

We now give two properties of Cayley graphs with polynomial growth that we will need to prove our results.

\begin{lemma}
Let $\Gamma=(V,E)$ be a Cayley graph with polynomial growth of order $D$. Let $a,b\in\mathbb{R}^*_+$ and $B(x,aR)$ be a ball in $\Gamma$ of radius $aR$. Then $\exists N\in \mathbb{N}^*$ such that $B(x,aR)$ is the union of $N$ balls of radius $bR$ and this number does not depend on $R$. More precisely, we can take $N=\lceil C^2(\frac{2a+b}{b})^D\rceil$ where $C$ is a constant satisfying $C^{-1}n^D\leq V(n) \leq Cn^D$.
\label{lem}
\end{lemma}

\begin{proof} Let $\{y_i\}_{i=1}^{m}$ be a maximal subset of vertices in $B(x,aR)$ such that $d(y_i,y_j)\geq bR \text{ for } i\not = j$. Then $\cup_{i=1}^n B(y_i,bR)\supset B(x,aR)$ and, by the triangle inequality, $B(y_i,\frac{bR}{2})\cap B(y_j,\frac{bR}{2})=\emptyset$. This implies
\begin{align}
\sum_{i=1}^m |B(y_i,\frac{bR}{2})| \leq |B(x,(a+\frac{b}{2})R)|.
\label{lem1}
\end{align}
Since the graph has polynomial growth of order $D$, we know that there exists $C$ such that $C^{-1}n^D\leq |B(z,n)|\leq C n^D$  $\forall z\in V$. We approximate the volume of the balls in equation (\ref{lem1}) using the latter inequality and we obtain that $m\leq C^2(\frac{2a+b}{b})^D$.
\end{proof}

The second property is a discrete isoperimetric inequality.

\begin{prop}
Let $\Gamma=(V,E)$ be a Cayley graph with polynomial growth of order $D$. There exists $C$ such that for any finite subset $\Omega\subset V$, $\delta \Omega$ its boundary, and $\bar\Omega:=\Omega\cup\delta\Omega$, we have that
\begin{align}
\frac{|\bar\Omega|^{\frac{(D-1)}{D}}}{|\delta\Omega|}\leq C.
\end{align}
\label{prop1}
\end{prop}

For the proof of this proposition, we refer to \cite{CS-C}. In fact, the result that we give corresponds to the first particular case of Theorem $1$ of \cite{CS-C}, but formulated in the setting of Cayley graphs.

\end{section}

\begin{section}{Isoperimetric upper bound for $\sigma_1$ in Cayley graphs with polynomial growth}

In this section, we prove the results presented in the introduction and give examples of application.

The following variational characterization of the Steklov eigenvalues on graphs with boundary is important for the proof of our main result, Theorem 1.

\begin{align}
\sigma_j=\min_E \max_{v\in E,v\not =0}R(v),
\label{sigmak}
\end{align}
where $E$ is the set of all linear subspaces of $\mathbb{R}^V$ of dimension $j+1$, and $R(v)$ is the Rayleigh quotient associated to the Dirichlet-to-Neumann operator (see \cite{hua2017})

\[
R(v):=\frac{\sum_{i\sim j}(v(i)-v(j))^2}{\sum_{i\in B}v(i)^2}.
\]

\begin{subsection}{Proof of Theorem \ref{thm1}}

The proof consists of finding two regions of the graph with boundary with a sufficient number of vertices of the boundary, then building test functions, evaluating their Rayleigh quotient, and using the variational characterization in order to obtain an upper bound for $\sigma_1$.

\begin{proof}
By Lemma \ref{lem}, there exists $c_1$ such that a ball of radius $3R$ in $\Gamma$ is the union of $c_1$ balls of radius $\frac{1}{2}R$.

If $|B|\leq c_1+1$, it is easy to show that the result is true using that $\sigma_1$ is bounded from above by $d$, the degree of the host Cayley graph: we have that

\[
\sigma_1\leq d=\frac{d|B|}{|B|}\leq \frac{d(c_1+1)}{|B|}=:c_2\frac{1}{|B|}
\]

and, if $D\geq 2$,

\[
\sigma_1\leq c_2\frac{1}{|B|} \leq c_2\frac{|V'|^{\frac{D-2}{D}}}{|B|}.
\]

From now on, we will assume $|B|>c_1+1$. We define

\[
\alpha:=\frac{|B|}{c_1+1}.
\]

Let $x\in V$. We set

\[
r_x:=\min\{r \in \mathbb{N}: |B(x,r)\cap B|\geq\alpha\}
\]
 and
 
 \[
R:=\min_{x\in V}r_x.
\]

Then, we have that $\forall x \in V$, $|B(x,R-1)\cap B|<\alpha$ and there exists $x_0 \text{ such that } |B(x_0,R)\cap B|\geq \alpha$. We remark that $R\geq 1$. Since $B(x,R-1)\geq B(x,\frac{1}{2}R)$ we have that $B(x_0,3R)$ is the union of $c_1$ balls of radius $R-1$. This implies

\[
|B(x_0,3R)\cap B|<c_1 \alpha
\]

and consequently

\begin{align*}
|B(x_0,3R)^c\cap B|&=|B|-|B(x_0,3R)\cap B|\\
&>|B|-c_1\alpha\\
&=|B|-c_1\frac{|B|}{c_1+1}\\
&=\frac{|B|}{c_1+1}=\alpha.
\end{align*}

Hence, we have found two regions, $B(x_0,R)$ and $B(x_0,3R)^c$, such that

\[
|B(x_0,R)\cap B|\geq\alpha
\]

and 

\[
|B(x_0,3R)^c\cap B|>\alpha.
\]

We define two test functions, one with support $B(x_0,2R)$, and the other with support $B(x_0,2R)^c$.

\[f_1(y)=
\begin{cases}
1 &\text{if }y\in B(x_0,R)\\
1-\frac{k}{R} &\text{if } k:=d(y,B(x_0,R))\leq R\\
0 &\text{otherwise},
\end{cases}
\]

\[f_2(y)=
\begin{cases}
1 &\text{if }y\in B(x_0,3R)^c\\
1-\frac{k}{R} &\text{if } k:= d(y,B(x_0,3R)^c)\leq R\\
0 &\text{otherwise}.
\end{cases}
\]

We consider the linear subspace $W$ of $\mathbb{R}^V$ generated by $f_1$ and $f_2$. The variational characterization of equation (\ref{sigmak}) gives 

\[
\sigma_1\leq \max_{v\in W}R(v).
\]

Since $f_1$ and $f_2$ have disjoint support, it implies

\[
\sigma_1\leq \max\{R(f_1),R(f_2)\}.
\]

$R(f_1)$ can be evaluated in the following way. The denominator is

\[
\sum_{i\in B}f_1(i)^2\geq|B(x_0,R)\cap B|\geq\alpha=\frac{|B|}{c_1+1}.
\]

The only edges contributing to the sum in the numerator $\sum_{i\sim j}(f_1(i)-f_1(j))^2$ are the ones in $B(x_0,2R)\setminus B(x_0,R)$. In this annulus, for two adjacent vertices, we have that $(f_1(i)-f_1(j))^2\leq\frac{1}{R^2}$. Moreover, the number of edges in this annulus is smaller than or equal to the number of edges in $B(x_0,3R)$.  Hence we have

\[
\sum_{i\sim j}(f_1(i)-f_1(j))^2\leq\sum_{i\sim j, i,j\in B(x_0,3R)}\frac{1}{R^2}.
\]

Because the graph has polynomial growth of order $D$, there exists $c_3>0 $ such that $|B(x_0,3R)|\leq c_3 (3R)^D$. We recall that the graph is the Cayley graph defined by a group $G$ and a generating set $S$ of $G$. The degree of the graph is $|S|=|B(y,1)|\leq c_3$.
By the handshaking lemma, $|E(B(x_0,3R),B(x_0,3R))|\leq \frac{1}{2} |B(x_0,3R)||S|\leq \frac{1}{2}c_3^2(3R)^D:=c_4 R^D$. Consequently, for $D=1$ or $D=2$, we have

\[
\sum_{i\sim j, i,j\in B(x_0,3R)}\frac{1}{R^2}\leq c_4\frac{R^D}{R^2}\leq c_4
\]

and the Rayleigh quotient of $f_1$ becomes 

\[
R(f_1)=\frac{\sum_{i\sim j}(f_1(i)-f_1(j))^2}{\sum_{i\in B}f_1(i)^2}\leq\frac{(c_1+1)c_4}{|B|}=:\frac{c_5}{|B|}
\]

If $D\geq 2$, we note that we have the following equality

\[
\sum_{i\sim j, i,j\in B(x_0,3R)}\frac{1}{R^2}= (\sum_{i\sim j, i,j\in B(x_0,3R)}\frac{1}{R^D})^{\frac{2}{D}}(\sum_{i\sim j, i,j\in B(x_0,3R)}1)^{\frac{D-2}{D}}
\]

The left factor is bounded by a constant:

\[
(\sum_{i\sim j, i,j\in B(x_0,3R)}\frac{1}{R^D})^{\frac{2}{D}}\leq c_4^{\frac{2}{D}}=:c_6.
\]

For the right factor, we have

\[
(\sum_{i\sim j, i,j\in B(x_0,3R)}1)^{\frac{D-2}{D}}\leq (\frac{c_3}{2}|V'|)^{\frac{D-2}{D}},
\]

and we obtain

\[
\sum_{i\sim j, i,j\in B(x_0,3R)}\frac{1}{R^2}\leq c_6(\frac{c_3}{2})^{\frac{D-2}{D}}|V'|^{\frac{D-2}{D}}=:c_7|V'|^{\frac{D-2}{D}}.
\]

Hence, if $D\geq 2$, the numerator of the Rayleigh quotient satisfies

\[
\sum_{i\sim j}(f_1(i)-f_1(j))^2\leq c_7|V'|^{\frac{D-2}{D}}.
\]

The Rayleigh quotient of $f_1$ becomes 

\[
R(f_1)=\frac{\sum_{i\sim j}(f_1(i)-f_1(j))^2}{\sum_{i\in B}f_1(i)^2}\leq\frac{(c_1+1)c_7|V'|^{\frac{D-2}{D}}}{|B|}=:c_8\frac{|V'|^{\frac{D-2}{D}}}{|B|}.
\]

By the definition of the test functions, the same upper bound can be obtained for $f_2$. We conclude that

\[
\sigma_1\leq \max\{R(f_1),R(f_2)\}\leq
\begin{cases}
c_5\frac{1}{|B|}&\text{if }D\leq2,\\
c_8\frac{|V'|^{\frac{D-2}{D}}}{|B|}&\text{if }D\geq2.
\end{cases}
\]

In order to unify the case $|B|\leq c_1+1$ and the general case, we take $C:=\max\{c_5,c_2\}$ if $D\leq2$ or $C:=\max\{c_8,c_2\}$ if $D\geq 2$. This completes the proof.

\end{proof}

\begin{remark}
The proof is qualitative rather than quantitative since the goal here is not to find an optimal constant (the constant depends on the generating set of the group).
\end{remark}

\begin{exe}
An example of a group with polynomial growth of order $D$ is $\mathbb{Z}^D$.
\end{exe}

\begin{exe}
The Heisenberg group over $\mathbb{Z}$,

\[
Heis(\mathbb{Z})=\left\{\begin{pmatrix}
1 & x & z \\
0 & 1 & y \\
0 & 0 & 1
\end{pmatrix}: x,y,z \in \mathbb{Z}\right\},
\]
is an example of a group with polynomial growth of order $4$, which is not quasi-isometric to $\mathbb{Z}^4$ (on this affirmation, see \cite{GHarpe}, p. 13, for example). Hence, for the Steklov problem on a graph with boundary $(\Gamma'=(V',E'),B)$ included in a Cayley graph associated to the Heisenberg group, $\sigma_1$ is bounded from above by $C(Heis(\mathbb{Z}))\frac{|V'|^{1/2}}{|B|}$.
\label{exHeis}
\end{exe}

\begin{exe}
An important theorem due to M. Gromov characterizes finitely generated groups of polynomial growth (see \cite{Gr81}).  It says that that a group is of polynomial growth if and only if it has a nilpotent subgroup of finite index. Lattices in nilpotent Lie groups, which are finitely generated and themselves nilpotent are other examples where the theorem holds (for the existence of such lattices, see e.g. \cite{raghunathan} and \cite{Eberlein}).
\label{exGrom}
\end{exe}

\begin{remark}
Given a Cayley graph $\Gamma=(V,E)$ with polynomial growth of order $1$ or $2$, Theorem \ref{thm1} shows that for a sequence $\{(\Gamma'_n,B_n)\}_{n\in\mathbb{N}}$ of graphs with boundary included in $\Gamma$ and satisfying $|B_n| \rightarrow \infty$, we have that $\sigma_1$ tends to $0$ as $n$ tends to infinity.
\end{remark}

\end{subsection}

\begin{subsection}{Application to subgraphs}

\begin{proof}[Proof of Corollary \ref{cor1}]
By the isoperimetric inequality in Proposition \ref{prop1}, there exists $c_1>0$ such that $\frac{|\bar\Omega|^{\frac{(D-1)}{D}}}{|\delta\Omega|}\leq c_1$, where $\bar\Omega=\delta\Omega\cup\Omega$. We raise the latter inequality to the power of $\frac{D-2}{D-1}$ and obtain $|\bar\Omega|^{\frac{D-2}{D}}\leq (c_1|\delta\Omega|) ^{\frac{D-2}{D-1}}=:c_2 |\delta\Omega| ^{\frac{D-2}{D-1}}$. By Theorem \ref{thm1}, there exists $c_3$ such that $\sigma_1\leq c_3\frac{|\bar\Omega|^{\frac{D-2}{D}}}{|\delta \Omega|}$. Consequently,

\[
\sigma_1\leq c_3\frac{|\bar\Omega|^{\frac{D-2}{D}}}{|\delta\Omega|}\leq c_3 c_2\frac{|\delta\Omega|^{\frac{D-2}{D-1}}}{|\delta\Omega|}= c_3 c_2 \frac{1}{|\delta\Omega|^{\frac{1}{D-1}}}=:c_4\frac{1}{|\delta\Omega|^{\frac{1}{D-1}}}.
\]
\end{proof}

\begin{remark}
For $D=1$, we remark that by Theorem \ref{thm1}, we have that $\sigma_1(\Omega)\leq C(\Gamma)\frac{1}{|\delta\Omega|}$.
\end{remark}

\begin{proof}[Proof of Corollary \ref{cor2}]
By the isoperimetric inequality in Proposition \ref{prop1}, there exists $c_1>0$ such that $\frac{|\bar\Omega|^{\frac{(D-1)}{D}}}{|\delta\Omega|}\leq c_1$. By Theorem \ref{thm1}, there exists $c_2$ such that $\sigma_1\leq c_2\frac{|\bar\Omega|^{\frac{D-2}{D}}}{|\delta \Omega|}$. Hence, we have

\[
\sigma_1\leq c_2\frac{|\bar\Omega|^{\frac{D-2}{D}}}{|\delta\Omega|}=c_2\frac{|\bar\Omega|^{\frac{D-1}{D}}|\bar\Omega|^{\frac{-1}{D}}}{|\delta\Omega|}\leq c_2 c_1 |\bar\Omega|^{\frac{-1}{D}}=:c_3\frac{1}{|\bar\Omega|^{\frac{1}{D}}}.
\]
\end{proof}

\begin{remark}
Since $\bar\Omega=\delta\Omega\cup\Omega$, we also have $\sigma_1\leq C(\Gamma)\frac{1}{|\Omega|^{\frac{1}{D}}}$ and $\sigma_1\leq C(\Gamma)\frac{1}{|\delta\Omega|^{\frac{1}{D}}}$ but this last bound is weaker than Corollary \ref{cor1}.
\end{remark}

\begin{remark}
In a Cayley graph with polynomial growth of order $D\geq2$, for a sequence $\{\Omega_n\}_{n\in\mathbb{N}}$ of finite subsets satisfying $|\Omega_n| \rightarrow \infty$, we have that $\sigma_1(\Omega_n)$ tends to $0$ as $n$ tends to infinity. 
\end{remark}

\begin{remark} For graphs with boundary induced by a finite subset of $\mathbb{Z}^n$, the result of Corollary \ref{cor2} was recently obtained by Han and Hua (see Corollary 1.4 in \cite{hua2019steklov}), who also give an explicit constant.
\end{remark}

\end{subsection}

\end{section}

\begin{paragraph}{Acknowledgements}
I would like to thank the anonymous referee for the helpful comments and suggestions to improve the manuscript.
\end{paragraph}

\bibliography{biblio}{}
\bibliographystyle{plain}

\bigskip
\footnotesize

\textsc{Université de Neuchâtel, Institut de mathématiques, Rue Emile-Argand 11, CH-2000 Neuchâtel, Switzerland}\par\nopagebreak
\textit{E-mail address}, \texttt{helene.perrin@unine.ch}

\end{document}